\documentclass[11pt, a4paper]{amsproc}

\usepackage[hmargin=3.2cm,vmargin=3.4cm]{geometry}
\usepackage{amsmath,amssymb}
\usepackage{mathtools}
\usepackage{enumitem}
\usepackage{hyperref}
\usepackage{todonotes}

\numberwithin{equation}{section}

\newtheorem{theorem}{Theorem}[section]
\newtheorem{proposition}[theorem]{Proposition}
\newtheorem{corollary}[theorem]{Corollary}
\newtheorem{lemma}[theorem]{Lemma}
\newtheorem{theoremA}{Theorem}

\newtheorem{corollaryB}{Corollary}

\theoremstyle{definition}
\newtheorem{definition}[theorem]{Definition}
\newtheorem{remark}[theorem]{Remark}

\newcommand{\norm}[1]{\left\|#1 \right\|}

\title{Remarks on approximability and stability for groups}

\author{Vadim Alekseev}
\address{Vadim Alekseev, TU Dresden, 01062 Dresden, Germany}
\email{vadim.alekseev@tu-dresden.de}

\author{Andreas Thom}
\address{Andreas Thom, TU Dresden, 01062 Dresden, Germany}
\email{andreas.thom@tu-dresden.de}

\date{\today}
\begin{document}

\begin{abstract}
In this paper, we provide several instances in which interesting approximation and stability properties are inherited by quotients with respect to finitely generated normal subgroups or, more strongly, normal subgroups with Kazhdan’s property (T). Applications arise when these observations are combined with variations of the Rips construction due to Wise and Belegradek--Osin.
\end{abstract}

\maketitle

\tableofcontents

\section{Introduction}

The interplay between geometric group theory, approximation and stability
properties, and non-commutative harmonic analysis produces a vast range of phenomena.
A group is called sofic if it admits asymptotic permutation models \cite{PestovBriefGuide} and is called stable in finite actions \cite{GohlaThom24} if any sofic approximation is weakly contained in finite actions in the sense of Kechris \cite{Kechris12}.
Flexible P-stability formalizes the possibility of correcting asymptotic
permutation models \cite{BL20}, while residual finite-dimensionality (RFD) of the full
group $C^*$-algebra records how effectively a group can be detected by
finite-dimensional unitary representations, see \cite{BrownOzawa}.
Gromov hyperbolic groups sit at an interface where these themes can be combined
with variants of the Rips construction, see \cite{BelegradekOsin08, Rip82, WiseRipsRF}. This way arbitrary finitely presented groups can be represented as quotients of hyperbolic groups by controlled normal subgroups.
Throughout we consider short exact sequences of countable groups
\begin{equation}\label{eq:ses}
1 \longrightarrow N \longrightarrow G \xrightarrow{\pi} Q \longrightarrow 1,
\end{equation}
so that $Q=G/N$ is the quotient.
When $N$ is finitely generated, approximate local triviality of $N$ in the Hamming
metric forces $N$ to act trivially on a large subset.
When $N$ is Kazhdan, approximate invariance of a finite number of atoms in a p.m.p.\ action or a finite number of vectors in a unitary representation
can be promoted to genuine invariance.
These elementary ingredients yield the following implications for
extensions with finitely generated resp.\ Kazhdan normal subgroup.

\begin{theoremA}\label{thm:A}
Let
\[
1\longrightarrow N\longrightarrow G\xrightarrow{\pi} Q\longrightarrow 1
\]
be a short exact sequence with $N$ finitely generated. Then:
\begin{enumerate}[label=$(\roman*)$]
\item if $G$ is flexibly P-stable, then $Q$ is flexibly P-stable;
\item if $N$ is Kazhdan, $G$ is residually finite, stable in finite actions and $Q$ is
sofic, then $Q$ is residually finite;
\item  if $N$ is Kazhdan and $G$ is RFD, then $Q$ is RFD;
\item if $N$ is Kazhdan and the canonical p.m.p.~action
$G\curvearrowright\{0,1\}^Q$ is sofic, then $Q$ is sofic.
\end{enumerate}
\end{theoremA}

The proofs appear in Sections~\ref{sec:flex}--\ref{sec:sofic-shift}. The possibility of
the first implication was anticipated in Becker--Lubotzky, see \cite{BL20}. The proof of the 
fourth implication also depends on results of Kun, see \cite{KunSoficT}.
As a first application we record four corollaries based on the
residually finite Rips construction of Wise \cite{WiseRipsRF} and the
refined version due to Belegradek--Osin \cite{BelegradekOsin08}.

\setcounter{corollaryB}{1}

\begin{corollaryB}\label{cor:B}
There exists a hyperbolic, residually finite group that is not flexibly
P-stable.
\end{corollaryB}

To the best of our knowledge, no example of such a group was previously known. Without
residual finiteness, this result and a similar strategy was outlined 
by Becker--Lubotzky, see \cite{BL20}.

\begin{corollaryB}\label{cor:C}
There exists a hyperbolic group that cannot be stable in finite actions,
provided it is residually finite.
\end{corollaryB}

It is a major open problem whether every hyperbolic group is residually
finite -- see, for instance, \cite{WiseRipsRF} for background.
On the other hand, it is also open whether there exists a residually finite
group that is not stable in finite actions; see the work of Gohla and the second author \cite{GohlaThom24}. The previous corollary shows that one of the problems has a negative answer.

\begin{corollaryB}\label{cor:D}
There exists a hyperbolic group $G$, which is not RFD.
\end{corollaryB}

For hyperbolic groups acting properly discontinuously and cocompactly on trees or on the hyperbolic plane, the universal group $C^*$-algebra is known to be RFD; see, for instance, \cite[Section~2.39]{BrownOzawa}. For fundamental groups of closed hyperbolic $3$-manifolds, residual finite dimensionality of the full group $C^*$-algebra also holds; see Remark~\ref{rem:3dRFD} below. It remains open whether lattices in $\mathrm{SO}(n,1)$ or $\mathrm{SU}(n,1)$ are RFD for $n \ge 4$, see \cite[Conjecture 6.4]{LuSh04}. Note that as a consequence of a negative solution of the Connes Embedding Problem \cite{mipre}, the group $F_2 \times F_2$ and any group containing it is not RFD.

\begin{corollaryB}\label{cor:E}
If a non-sofic group exists, then there exists a hyperbolic group admitting a
non-sofic p.m.p.\ action.
\end{corollaryB}

Again, this is somewhat surprising, since the general belief is that there do exist non-sofic groups, while hyperbolic 2- and 3-manifold groups do have the property that all actions are sofic. Indeed, those groups are known to have property MD of Kechris \cite{Kechris12} and thus all actions are weakly contained in finite actions; in particular, they are all sofic. It is an outstanding open problem to decide if a non-sofic group or at least a non-sofic p.m.p.\ action exists.

\medskip

The paper is organized as follows.
Section~\ref{sec:flex} reviews flexible P-stability and records the quotient
statement.
Section~\ref{sec:sfa} does the same for stability in finite actions in the
presence of a Kazhdan normal subgroup.
Section~\ref{sec:RFD} discusses residual finite-dimensionality. In Section~\ref{sec:meassof} we discuss measured sofic approximations, a useful tool for the next section.
Section~\ref{sec:sofic-shift} observes that soficity of the quasi-regular Bernoulli
shift forces soficity of the quotient.
Section~\ref{sec:Rips} recalls the relevant variations of the Rips construction and proves
Theorem~\ref{thm:A} and its corollaries.

\section{Flexible P-stability}\label{sec:flex}

Let's recall the basic setup of soficity and permutation stability. The notion of soficity was introduced by Gromov \cite{Gromov99}, see also \cite{PestovBriefGuide, Weiss00}.
Let $X$ be a finite set. For $\sigma\in{\rm Sym}(X)$ write
\[
\ell_X(\sigma)=\frac{1}{|X|}\,|\{x\in X\mid \sigma(x)\neq x\}|
\qquad\text{and}\qquad
d_X(\sigma,\tau)=\ell_X(\sigma^{-1}\tau).
\]

\begin{definition}
\label{def:hamming-sep}
Let $S$ be a finite set and $F=F(S)$ the free group.
Given $\pi\colon F\twoheadrightarrow Q$ with kernel $N$,
a sequence of homomorphisms
\[
\varphi_n\colon F\to {\rm Sym}(k_n)
\]
\emph{separates $N$ in the Hamming metric} if
\[
\lim_{n\to\infty}\ell_{k_n}(\varphi_n(w))=
\begin{cases}
0 &\text{if } w\in N,\\
1 &\text{if } w\notin N.
\end{cases}
\]
\end{definition}

It is a well-known characterization to say that a group $Q=F/N$ is \emph{sofic} if and only if there exists a sequence of homomorphisms as above that separate $N$ in the Hamming metric. 

The Hamming metric enjoys various abstract properties that are not consequence of the triangle inequality. Apart from being conditionally negative definite on ${\rm Sym}(X)$, the following elementary lemma contains maybe the most striking abstract property of the Hamming metric.

\begin{lemma}
\label{lem:small-subgroup}
Let $\sigma_1,\dots,\sigma_m\in{\rm Sym}(X)$ satisfy
$\ell_X(\sigma_i) \leq \varepsilon$ for all $i$.
Then every $\tau\in\langle\sigma_1,\dots,\sigma_m\rangle$ satisfies $\ell_X(\tau)\le m\varepsilon.$
\end{lemma}

\begin{proof}
Let ${\rm supp}(\sigma_i)=\{x\in X\mid \sigma_i(x)\neq x\}$ and set
$\Omega=\bigcup_{i=1}^m{\rm supp}(\sigma_i)$. Then
$|\Omega| \leq m\varepsilon|X|$.
Every $\sigma_i$ fixes $\Omega^c$ pointwise and hence so does every word in the
$\sigma_i^{\pm1}$.
Thus each $\tau$ moves points only in $\Omega$, and the claim follows.
\end{proof}

\begin{lemma}
\label{lem:kernel-support}
Assume that $N$ is generated by $m$ elements $n_1,\dots,n_m$.
Let $\varphi\colon G\to{\rm Sym}(X)$ be a homomorphism such that
$\ell_X(\varphi(n_i))\leq\varepsilon$ for $i=1,\dots,m$.
Then there exists $\Omega\subseteq X$ with $|\Omega|\leq m\varepsilon|X|$
such that $\varphi(N)$ fixes $\Omega^c$ pointwise.
\end{lemma}
\begin{proof}
Apply Lemma~\ref{lem:small-subgroup} to
$\varphi(n_1),\dots,\varphi(n_m)$.
\end{proof}

Let us recall also the notion of asymptotic homomorphism and flexible stability in this setting.

\begin{definition}
\label{def:asymp-hom}
Let $G$ be a countable group.
A sequence of maps $\sigma_n\colon G\to{\rm Sym}(X_n)$
is an \emph{asymptotic homomorphism} if
\[
\lim_{n\to\infty} d_{X_n}(\sigma_n(gh),\sigma_n(g)\sigma_n(h))=0
\qquad \forall g,h\in G.
\]
\end{definition}

\begin{definition}
\label{def:flexP}
A finitely generated group $G$ is \emph{flexibly P-stable} if
for every asymptotic homomorphism
$\sigma_n\colon G\to{\rm Sym}(X_n)$
there exist finite sets $Y_n$ with $X_n\subseteq Y_n$,
\[
\frac{|Y_n\setminus X_n|}{|X_n|}\to 0,
\]
and homomorphisms
$\tau_n\colon G\to{\rm Sym}(Y_n)$ such that
\[
\lim_{n\to\infty} d_{X_n}(\sigma_n(g),\tau_n(g)|_{X_n})=0
\qquad \forall g\in G.
\]
\end{definition}

The notion of flexible P-stability and usual P-stability was introduced and studied in a series of papers, including \cite{AP14, BL20, BLT18}.

\begin{theorem}
\label{thm:flex-quot}
Let
\[
1\longrightarrow N\longrightarrow G\xrightarrow{\pi} Q\longrightarrow 1
\]
be a short exact sequence with $N$ finitely generated.
If $G$ is flexibly P-stable, then $Q$ is flexibly P-stable.
\end{theorem}

\begin{proof}
Let $\sigma_n\colon Q\to{\rm Sym}(X_n)$ be an asymptotic homomorphism and set
\[
\bar{\sigma}_n := \sigma_n\circ\pi\colon G\to{\rm Sym}(X_n).
\]
Then $(\bar{\sigma}_n)_n$ is an asymptotic homomorphism of $G$.
By flexible P-stability of $G$ we find finite sets $Y_n\supseteq X_n$ with
$|Y_n\setminus X_n|/|X_n|\to 0$ and homomorphisms
\[
\tau_n\colon G\to{\rm Sym}(Y_n)
\]
such that
\[
d_{X_n}\bigl(\bar{\sigma}_n(g),\tau_n(g)|_{X_n}\bigr)\to 0
\qquad \forall g\in G.
\]

Let $n_1,\dots,n_m$ generate $N$.
Since $\pi(n_i)=e$, the restrictions $\tau_n(n_i)|_{X_n}$ are asymptotically
trivial.
Applying Lemma~\ref{lem:kernel-support} to the restricted action on $X_n$,
we obtain subsets $\Omega_n\subseteq X_n$ with $|\Omega_n|/|X_n|\to 0$
such that $N$ fixes $X_n\setminus\Omega_n$ pointwise.

Consider the induced action of $G$ on the finite orbit set $Y_n/N$.
Since $N$ is normal, this action factors through $Q$.
Moreover, a proportion tending to $1$ of points in $X_n$ correspond to singleton
$N$-orbits. Equivalently, there exists a sequence $\delta_n\to 0$ and subsets
$X_n'\subseteq X_n$ with $|X_n'|\ge (1-\delta_n)|X_n|$ such that every point of
$X_n'$ has trivial $N$-stabilizer. On $X_n'$ the induced $Q$-action agrees with
the original maps $\sigma_n$ up to $\delta_n$ in Hamming distance, i.e.
\[
d_{X_n}\bigl(\sigma_n(q),\,\bar\tau_n(q)\bigr)\le \delta_n\qquad \forall q\in Q,
\]
where $\bar\tau_n$ denotes the permutation action of $Q$ induced by
$\tau_n$ on $Y_n/N$ and restricted to $X_n$ via the identification of
singleton $N$-orbits with their unique elements.
Extending $X_n$ to $Y_n/N$ alters cardinality by a vanishing proportion, so the
induced $Q$-actions on $Y_n/N$ correct the asymptotic homomorphism $(\sigma_n)_n$ to a genuine action.
\end{proof}

\begin{corollary}
\label{cor:flex-sofic-RF}
In the situation of the theorem. If $G$ is flexibly P-stable and $Q$ is sofic, then $Q$ is residually finite.
\end{corollary}

\begin{proof}
By Theorem~\ref{thm:flex-quot}, $Q$ is flexibly P-stable.
For sofic groups, flexible P-stability yields residual finiteness; see \cite{BL20}.
\end{proof}

\section{Stability in finite actions}\label{sec:sfa}

We now turn to the weaker stability notion of Gohla and the second author.
The argument below employs the strategy that residual finiteness of $G$
allows one to amplify an asymptotic homomorphism of $Q$
to a sofic approximation of $G$ without losing the partition statistics. Let's recall the notion of stability in finite actions {\cite[Definition 3.7]{GohlaThom24}}:

\begin{definition}
\label{def:sfa}
A countable group $G$ is \emph{stable in finite actions} if,
for every sofic approximation of $G$, the corresponding limit action
is weakly contained in the class of finite $G$-actions.
\end{definition}

For the notion of weak containment see \cite{Kechris12}. We will need the following two lemmas.

\begin{lemma}
\label{lem:amplify-sofic}
Let $G$ be residually finite and let
$\alpha_n\colon G\to{\rm Sym}(X_n)$ be an asymptotic homomorphism.
Then there exist finite quotients $\rho_n\colon G\to H_n$ and maps
\[
\theta_n\colon G\to{\rm Sym}(X_n\times H_n),
\qquad
\theta_n(g)(x,h):=\bigl(\alpha_n(g)(x),\ \lambda_{H_n}(\rho_n(g))(h)\bigr),
\]
such that $(\theta_n)_n$ is a sofic approximation of $G$.
Moreover, the action of $G$ on $X_n$ induced by $\alpha_n$
is a factor of the action induced by $\theta_n$ via the coordinate projection
$X_n\times H_n\to X_n$.
\end{lemma}

\begin{proof}
Fix an increasing exhaustion $F_1\subset F_2\subset\cdots$ of $G$ by finite
sets with $e\in F_n$.
Using residual finiteness, choose $\rho_n\colon G\to H_n$ such that
$\rho_n$ is injective on $F_n$.

Since $\lambda_{H_n}\circ\rho_n$ is a genuine homomorphism,
the asymptotic multiplicativity of $\alpha_n$ implies that $\theta_n$
is an asymptotic homomorphism.
For $g\in F_n\setminus\{e\}$ we have $\rho_n(g)\ne e$, hence
$\lambda_{H_n}(\rho_n(g))$ moves every element of $H_n$.
Consequently $\theta_n(g)$ moves every element of $X_n\times H_n$ and thus
\[
\ell_{X_n\times H_n}(\theta_n(g))=1
\qquad\text{for all } g\in F_n\setminus\{e\}.
\]
Therefore $(\theta_n)$ is a sofic approximation of $G$, as desired.
Finally, the coordinate projection $X_n\times H_n\to X_n$ is $G$-equivariant
for the actions defined by $\theta_n$ and $\alpha_n$, so the latter is a factor
of the former.
\end{proof}

\begin{lemma}
\label{lem:T-upgrade-atoms}
Let $N$ be a Kazhdan group with a finite generating set
$K\subset N$. For every $d\in\mathbb{N}$ there exists $C>0$, which depends only on $K$ and $d$,
such that for every $\delta>0$ the following holds: whenever $N$ acts on a finite set $Y$ and
$Y=B_1\sqcup\cdots\sqcup B_d$ is a partition satisfying
\[
|kB_i\triangle B_i|<\delta |Y|
\qquad \forall\, k\in K,\ i=1,\dots,d,
\]
there exists a partition $Y=B_1'\sqcup\cdots\sqcup B_d'$
into $N$-invariant sets such that
\[
|B_i\triangle B_i'|\le C\,\sqrt{\delta} |Y|
\qquad \forall\, i=1,\dots,d.
\]
\end{lemma}

\begin{proof}
Since $N$ is Kazhdan and $K$ is a generating set, there exists
$\kappa>0$ such that for every unitary
representation $\pi:N\to \mathcal{U}(\mathcal{H})$ on a Hilbert space
$\mathcal{H}$ one has for every $v\in\mathcal{H}$
\begin{equation}\label{eq:distance-to-invariants}
\|v-Pv\|\;\le\;\kappa\,\max_{k\in K}\|\pi(k)v-v\|.
\end{equation}
Here $P:\mathcal{H}\to\mathcal{H}^N$ is the orthogonal projection onto the invariant subspace.

Consider the Hilbert space $\mathcal{H}:=\ell^2(Y,\mu)$, where $\mu$ is the normalized counting measure on $Y$. The $N$-invariant vectors are exactly the functions that are constant on
each $N$-orbit in $Y$. Write the orbit decomposition as
\[
Y=\bigsqcup_{\alpha} O_\alpha.
\]
Then the orthogonal projection $P:\mathcal{H}\to \mathcal{H}^N$ is given by
averaging on the orbits:
\[
(Pf)(y)\;=\;\frac1{|O_\alpha|}\sum_{z\in O_\alpha} f(z)
\qquad\text{for }y\in O_\alpha.
\]

Let now $Y=B_1\sqcup\cdots\sqcup B_d$ satisfy
\[
|kB_i\triangle B_i|<\delta |Y|
\qquad \forall\,k\in K,\ \forall\, i=1,\dots,d.
\]
For each $i$ set $f_i:=\mathbf{1}_{B_i}\in\ell^2(Y)$. For any $k\in N$ we
have
\[
\|\pi(k)f_i-f_i\|_2^2
=\frac{1}{|Y|}\sum_{y\in Y}\bigl|\mathbf{1}_{kB_i}(y)-\mathbf{1}_{B_i}(y)\bigr|^2
=\frac{1}{|Y|}\,|kB_i\triangle B_i|
<\delta,
\]
so that we can conclude
\begin{equation}\label{eq:fi-almost-invariant}
\max_{k\in K}\|\pi(k)f_i-f_i\|_2\;\le\;\sqrt{\delta}
\qquad\forall i=1,\dots,d.
\end{equation}

Applying \eqref{eq:distance-to-invariants} to each $f_i$ and using
\eqref{eq:fi-almost-invariant}, we obtain
\begin{equation}\label{eq:distance-fi-to-inv}
\|f_i-Pf_i\|_2\;\le\;\kappa\,\sqrt{\delta},
\qquad\forall i=1,\dots,d.
\end{equation}

For each orbit $O_\alpha$ and each $i$ define the density
\[
p_{i,\alpha}
\;:=\;\frac{|B_i\cap O_\alpha|}{|O_\alpha|}.
\]
Then by the explicit formula for $P$, we get
$(Pf_i)(y)=p_{i,\alpha}$ for all $y\in O_\alpha$.
Note that for each $\alpha$,
\[
0\le p_{i,\alpha}\le1,\qquad
\sum_{i=1}^d p_{i,\alpha}=1,
\]
since the sets $B_i$ form a partition of $Y$.
The norm $\|f_i-Pf_i\|_2^2$ can be computed orbitwise:
\begin{align*}
\|f_i-Pf_i\|_2^2
&=\frac1{|Y|}\sum_{\alpha}\sum_{y\in O_\alpha}
\bigl(\mathbf{1}_{B_i}(y)-p_{i,\alpha}\bigr)^2\\
&=\frac1{|Y|}\sum_{\alpha}|O_\alpha|\,p_{i,\alpha}(1-p_{i,\alpha}),
\end{align*}
because for $y\in O_\alpha$ the random variable
$\mathbf{1}_{B_i}(y)$ has mean $p_{i,\alpha}$ and variance
$p_{i,\alpha}(1-p_{i,\alpha})$.

Using \eqref{eq:distance-fi-to-inv} we obtain for every $i=1,\dots,d$
\begin{equation}\label{eq:S-def}
\|f_i-Pf_i\|_2^2
=  \sum_{\alpha}\frac{|O_\alpha|}{|Y|}
 p_{i,\alpha}(1-p_{i,\alpha})\le\; \kappa^2\delta.
\end{equation}

By Markov's inequality, the set of orbits $\alpha$ for which $p_{i,\alpha}(1-p_{i,\alpha})$ is bigger than $\sqrt{\delta}$ has total size at most $\kappa^2\sqrt{\delta}|Y|$.
Now fix $i\in\{1,\dots,d\}$ and define
\[
B_i' \;:=\; \bigsqcup_{\alpha:  p_{i,\alpha}\ge 1-2\sqrt{\delta}} O_\alpha.
\]
Then $B_i'$ is a union of $N$-orbits, hence $N$-invariant. Moreover, if
$O_\alpha\subset B_i'$ and $O_\alpha\subset B_j'$ with $i\ne j$, then
$p_{i,\alpha},p_{j,\alpha}\ge 1-2\sqrt{\delta}$, so
$p_{i,\alpha}+p_{j,\alpha}\ge 2-4\sqrt{\delta}>1$ for
${\delta}<\tfrac{1}{16}$, which is impossible since
$\sum_{k=1}^d p_{k,\alpha}=1$. Thus the sets $B_i'$ are pairwise disjoint.

Moreover, we have
\[
\|{Pf_i-\mathbf{1}_{B_i'}}\|_2^2 = \frac{1}{|Y|}\sum_{\substack{\alpha: p_{i,\alpha}< 1-2\sqrt{\delta}}} |O_\alpha| p_{i,\alpha}^2 \le \frac{1}{|Y|}\sum_{\substack{\alpha: p_{i,\alpha} \leq 2\sqrt{\delta}}} |O_\alpha|p_{i,\alpha}^2 + \kappa^2\sqrt{\delta} \le (2 + \kappa^2)\sqrt{\delta}.
\]
Therefore, we have
\begin{align*}
\frac{|B_i\triangle B_i'|^{1/2}}{|Y|^{1/2}} = \|{f_i-\mathbf{1}_{B_i'}}\|& \le \norm{f_i-Pf_i} + \|{Pf_i-\mathbf{1}_{B_i'}}\|
\\&\le \kappa\sqrt{\delta} + (2+\kappa^2)^{1/2}{\delta}^{1/4} \le 2(2+\kappa^2)^{1/2}{\delta}^{1/4},
\end{align*}
and hence, we obtain
\[
\frac{|B_i\triangle B_i'|}{|Y|} \le (8+4\kappa^2)\sqrt{\delta}.
\]
Finally, we observe that 
\[
|Y|-\sum_{i=1}^d |B_i'| \le \sum_{i=1}^d |B_i\triangle B_i'| \le d(8+4\kappa^2)\sqrt{\delta} |Y|.
\]
Therefore adding $Y\setminus \bigsqcup_{i=1}^d B_i'$ to $B_1'$ we obtain a partition of $Y$ into $N$-invariant sets satisfying the desired estimate.
\end{proof}

\begin{theorem}
\label{thm:quotient-sfa-kazhdan}
Assume that $G$ is residually finite and that $N$ is Kazhdan.
If $G$ is stable in finite actions, then $Q$ is stable in finite actions.
In particular, if $Q$ is sofic, then $Q$ is residually finite.
\end{theorem}

\begin{proof}
Let $\sigma_n\colon Q\to{\rm Sym}(X_n)$ be a sofic approximation of $Q$.
View it as an asymptotic homomorphism of $G$ by setting
\[
\alpha_n := \sigma_n\circ \pi\colon G\to{\rm Sym}(X_n).
\]
Since $\pi(N)=\{e\}$, the maps $\alpha_n$ are asymptotically trivial on $N$: $\ell_{X_n}(\alpha_n(g))\to 0$ for all $g\in N$.

By Lemma~\ref{lem:amplify-sofic}, using residual finiteness of $G$,
we may choose finite quotients $\rho_n\colon G\to H_n$ and obtain a
sofic approximation
\[
\theta_n\colon G\to{\rm Sym}(X_n\times H_n)
\]
whose associated actions admit the $G$-actions defined by $(\alpha_n)_n$
as factors via the coordinate projection.
In particular, the limit action of $(\alpha_n)_n$ is a factor of the
limit action of $(\theta_n)_n$.

Assume that $G$ is stable in finite actions.
Then the limit action of the sofic approximation $(\theta_n)_n$ is weakly
contained in the class of finite $G$-actions.
Since weak containment passes to factors, the same holds for the limit action
of $(\alpha_n)$.

Fix a finite set $E\subset Q$, an integer $d\ge 1$, and $\varepsilon>0$.
Choose a finite set $\bar E\subset G$ with $\pi(\bar E)=E$.
Let $K\subset N$ be a finite generating set.
Consider an arbitrary partition
\[
X_n=A_1\sqcup\cdots\sqcup A_d.
\]

The weak containment of the limit action of $(\alpha_n)$ in finite $G$-actions
means that for large $n$, the statistics of the partition $(A_i)_i$ under the
finite set $\bar E\cup K$ can be modelled by a genuine finite action of
$G$:
there exist a finite $G$-set $Y$, a homomorphism
$\tau\colon G\to{\rm Sym}(Y)$ and a partition $Y=B_1\sqcup\cdots\sqcup B_d$
such that
\[
\left|
\frac{|A_i\cap \alpha_n(g)A_j|}{|X_n|}
-
\frac{|B_i\cap \tau(g)B_j|}{|Y|}
\right|<\varepsilon
\qquad \forall\, g\in \bar E\cup K,\ 1\le i,j\le d.
\]

For $g\in K$, the $\ell(\alpha_n(g))\to 0$ as $n\to\infty$. Therefore 
\[
\sum_{i=1}^d \frac{|A_i\triangle (\alpha_n(g)A_i)|}{|X_n|} < \varepsilon
\]
for large enough $n$. Therefore
\[
\sum_{i=1}^d \frac{|B_i\triangle \tau(g)B_i|}{|Y|} < (d+1)\varepsilon
\]
and hence each $B_i$ will satisfy 
\[
\left|\frac{|B_i\triangle \tau(g)B_i|}{|Y|}\right| < (d+1)\varepsilon
\]
for $g\in K, 1\leq i \leq d$.
 
Applying Lemma~\ref{lem:T-upgrade-atoms}, we may replace $(B_i)$ by an
$N$-invariant partition $(B_i')$ such that
\[
\left|\frac{|B_i\triangle B_i'|}{|Y|}\right| < C'\sqrt{\varepsilon}
\]
for some $C'>0$ depending only on $d$ and the Kazhdan constant for $K\subset N$. Since $N\lhd G$, the induced action of $G$ on the finite Boolean algebra generated by the atoms $B_1',\dots,B_d'$ factors through $Q$.
Thus we obtain a finite $Q$-action whose partition statistics approximate those
of the partition $(A_i)_i$ under the set $E$, proving that $Q$ is stable in finite actions.
If $Q$ is sofic, then stability in finite actions implies residual finiteness, using results in \cite{GohlaThom24}.
\end{proof}

\section{Residually finite-dimensional groups}\label{sec:RFD}

A unital C$^*$-algebra $A$ is called \emph{residually finite-dimensional}
(\emph{RFD}) if the family of all finite-dimensional $^*$-representations of $A$
separates points. Equivalently,
\[
\bigcap_{\pi \in \mathrm{Rep}_{\mathrm{fd}}(A)} \ker(\pi)=\{0\},
\]
or, what is the same, $A$ embeds into a possibly infinite product of full
matrix algebras. This property can be viewed as a strong finite-dimensional
approximation property and plays a role analogous to residual finiteness for
discrete groups.

\begin{definition}
A countable group $G$ is called \emph{residually finite-dimensional (RFD)}
if its full group $C^*$-algebra $C^*(G)$ is residually finite-dimensional.
\end{definition}

Choi showed that the full group $C^*$-algebra of the free group $F_2$ is RFD \cite{Choi80}. On the positive side, Lubotzky and Shalom introduced and studied the representation-theoretic property~(FD) and established it for fundamental groups of closed surfaces; as a consequence, these groups are RFD full group $C^*$-algebras \cite{LuSh04}. For fundamental groups of closed hyperbolic $3$-manifolds, residual finite dimensionality of the full group $C^*$-algebra also holds; see Remark~\ref{rem:3dRFD}. A striking contrast appears for higher-rank arithmetic groups. Bekka proved
that if $\Gamma$ has the congruence subgroup property, then the full group
C$^*$-algebra $C^*(\Gamma)$ does \emph{not} admit a separating family of
finite-dimensional representations \cite{Bekka99}. In particular, $SL_3(\mathbb Z)$ is not RFD,
and more generally $SL_n(\mathbb Z)$ fails to be RFD for all $n\ge 3$.

We will use the following standard characterization (see, for instance,
\cite[Proposition~2.3.2]{BrownOzawa}): a full group $C^*$-algebra $C^*(G)$ is
residually finite-dimensional if and only if every unitary representation of
$G$ is weakly contained in the family of finite-dimensional unitary
representations of $G$.

\begin{proposition}
\label{prop:RFD-passes}
Assume that $N$ is Kazhdan.
If $G$ is RFD, then $Q$ is RFD.
\end{proposition}

\begin{proof}
Let $\pi\colon G\to Q$ be the quotient map.
Take an arbitrary unitary representation $\sigma$ of $Q$ and consider the
induced representation $\bar{\sigma}:=\sigma\circ\pi$ of $G$.
Assuming that $G$ is RFD, the representation $\bar{\sigma}$
is weakly contained in finite-dimensional unitary representations of $G$.

Fix a finite set $F\subset Q$ and $\varepsilon>0$, and let
$\bar F\subset G$ be a finite set of lifts. Let $K$ be a finite generating set for $N$.
Weak containment provides a finite-dimensional representation
$\rho\colon G\to U(d)$ and a unit vector $\xi$
whose matrix coefficients approximate those of $\bar{\sigma}$
on $\bar F \cup K$ up to $\varepsilon$.
In particular, since $\bar{\sigma}|_N$ is trivial, the vector $\xi$
is almost $K$-invariant.
By property {\rm (T)} of $N$,
the representation space of $\rho$ contains a nonzero $N$-invariant vector.
Let $V^N$ denote the $N$-fixed subspace.
Because $N\lhd G$, the space $V^N$ is $G$-invariant and the restricted representation
$\rho|_{V^N}$ factors through $Q=G/N$.
Compressing $\rho$ to $V^N$ yields a finite-dimensional representation of $Q$
whose matrix coefficients approximate those of $\sigma$ on $F$.
Since $F$ and $\varepsilon$ were arbitrary,
this shows that $\sigma$ is weakly contained in finite-dimensional
representations of $Q$.
Thus, we conclude that $Q$ is RFD.
\end{proof}

\begin{remark}\label{rem:3dRFD}
Let $M$ be a closed hyperbolic $3$-manifold and $G = \pi_1(M)$. It is known that $G$ has property~\textnormal{(FD)} in the sense of Lubotzky--Shalom. See for example Aschenbrenner--Friedl--Wilton \cite[(H.28) and (I.6)]{AschenbrennerFriedlWilton15}, where property~\textnormal{(FD)} for $G$ is deduced from work of Lubotzky--Shalom \cite{LuSh04} together with Agol's virtual fibering theorem and the Haglund--Wise theory of special cube complexes. Since property~\textnormal{(FD)} means that finite-image unitary representations are dense in the unitary dual, it immediately implies that the full group $C^*$-algebra $C^*(G)$ is residually finite-dimensional. In particular, for closed hyperbolic $3$-manifolds the universal group $C^*$-algebra is RFD.
\end{remark}

\section{Measured sofic approximations}

\label{sec:meassof}

Let $(X,\mu)$ be a finite probability measure space.
For permutations $\alpha,\beta\in\mathrm{Sym}(X)$, we set $
d_\mu(\alpha,\beta):=\mu\bigl(\{x\in X:\alpha(x)\neq\beta(x)\}\bigr)$. Recall also the total variation distance on probability measures: $
\|\nu-\nu'\|_{\mathrm{TV}}:=\sup_{A\subseteq X}|\nu(A)-\nu'(A)|.$ In some situations, a natural variation of soficity arises naturally, where one would not just act on a set but on a finite probability space and mistakes are measured by the metrics above --  definitions are straightforward. We do not spell this out in full detail since the following lemma allows to reduce such a \emph{measured} sofic approximation back to an ordinary one.

\begin{lemma}\label{lem:uniformize-measure}
Let $G$ be a group, $(X,\mu)$ be a finite measure space, $F \subset G$ be finite and $\sigma\colon F\to\mathrm{Sym}(X)$.
Assume that for some $\varepsilon\ge 0$:
\begin{enumerate}
\item[\textup{(a)}] For all $g,h\in F$ with $gh\in F$,
$
d_\mu\bigl(\sigma(gh),\sigma(g)\sigma(h)\bigr)\le \varepsilon.
$
\item[\textup{(b)}] For all $g\in F$,
$
\bigl\|\sigma(g)_{*}(\mu)-\mu\bigr\|_{\mathrm{TV}}\le \varepsilon.$
\end{enumerate}
Then for every integer $N\ge |X|$ there exist a finite set $\bar X$ with $|\bar X|=N$ and maps $p\colon\bar X\to X$, $\bar \sigma \colon F \to \mathrm{Sym}(\bar X)$
such that the following conditions hold:
\begin{enumerate}
\item[\textup{(i)}]
For every $g\in F$,
\[
\bigl|\bigl\{y\in\bar X:\ p(\bar\sigma(g)y)\neq \sigma(g)(p(y))\bigr\} \bigr|
\ \le\ N \varepsilon+|X|.
\]
\item[\textup{(ii)}]
For all $g,h\in F$ with $gh\in F$,
\[
d_{\bar X}\bigl(\bar\sigma(gh),\bar\sigma(g)\bar\sigma(h)\bigr)
\ \le\ 4\varepsilon+\frac{4|X|}{N}.
\]
\end{enumerate}
\end{lemma}

\begin{proof}
Choose integers $m_x\in\{0,1,\dots,N\}$ such that $\sum_{x\in X}m_x=N$ and
\[
\Bigl|\mu(x)-\frac{m_x}{N}\Bigr|\le \frac{1}{N}\qquad \forall x\in X.
\]
Define the discretized measure $\mu_N(x):=m_x/N$. Then
\[
\|\mu_N-\mu\|_{\mathrm{TV}}
=\frac12\sum_{x\in X}\Bigl|\mu_N(x)-\mu(x)\Bigr|
\ \le\ \frac{|X|}{2N}.
\]

Now, let $
\bar X:=\bigsqcup_{x\in X}\{x\}\times\{1,2,\dots,m_x\}$, and
$p(x,i):=x.$ Then $|\bar X|=N$ and $p_{*}(u)=\mu_N$, where $u$ denotes the uniform measure on $\bar X$

Fix $g\in F$. For each $x\in X$ set $r_x:=\min\{m_x,m_{\sigma(g)x}\}$ and define a partial bijection
\[
\psi_g(x,i):=(\sigma(g)x,i)\qquad\text{for }1\le i\le r_x.
\]
Because $\sigma(g)$ is a permutation of $X$, the images for different $x$ land in disjoint fibres,
so $\psi_g$ is injective. Its domain and image have the same cardinality $\sum_x r_x$.
Extend $\psi_g$ arbitrarily to a full permutation $\bar\sigma(g)\in\mathrm{Sym}(\bar X)$.

We set $
B_g:=\{y\in\bar X:\ p(\bar\sigma(g)y)\neq \sigma(g)(p(y))\}.$
By construction, $B_g$ consists exactly of the points not in the domain of $\psi_g$, hence
\[
|B_g|
=\sum_{x\in X}(m_x-r_x)
=\frac12\sum_{x\in X}\bigl|m_x-m_{\sigma(g)x}\bigr|.
\]
Dividing by $N$ gives
\[
u(B_g)
=\frac{1}{2}\sum_{x\in X}\bigl|\mu_N(x)-\mu_N(\sigma(g)^{-1}x)\bigr|
=\bigl\|\sigma(g)_{*}(\mu_N)-\mu_N\bigr\|_{\mathrm{TV}}.
\]
Using the triangle inequality and invariance of total variation under pushforward,
\begin{eqnarray*}
\bigl\|\sigma(g)_{*}(\mu_N)-\mu_N\bigr\|_{\mathrm{TV}}
&\le&
\bigl\|\sigma(g)_{*}(\mu_N)-\sigma(g)_{*}(\mu)\bigr\|_{\mathrm{TV}}
\\
&& +\bigl\|\sigma(g)_{*}(\mu)-\mu\bigr\|_{\mathrm{TV}} 
+\|\mu-\mu_N\|_{\mathrm{TV}} \\
&\le&
2\|\mu-\mu_N\|_{\mathrm{TV}}+\varepsilon \\
&\le& 
\varepsilon+\frac{|X|}{N},
\end{eqnarray*}
which proves (i). Now, fix $g,h\in F$ with $gh\in F$ and define the multiplicativity defect set
\[
A_{g,h}:=\{x\in X:\ \sigma(gh)x\neq \sigma(g)\sigma(h)x\}.
\]
By (a), $\mu(A_{g,h})\le\varepsilon$, hence
\[
\mu_N(A_{g,h})\le \mu(A_{g,h})+\|\mu_N-\mu\|_{\mathrm{TV}}
\le \varepsilon+\frac{|X|}{2N}.
\]
Now consider $y\in\bar X$ which simultaneously satisfies: $p(y)\notin A_{g,h}$, and $y\notin B_h$, $\bar\sigma(h)y\notin B_g$, and $y\notin B_{gh}$.
Then $p(\bar\sigma(h)y)=\sigma(h)p(y)$, and
$p(\bar\sigma(g)\bar\sigma(h)y)=\sigma(g)\sigma(h)p(y)$, and also
$p(\bar\sigma(gh)y)=\sigma(gh)p(y)$. Since $p(y)\notin A_{g,h}$, the last two coincide, and because
on the good parts our lifts act fibrewise by $(x,i)\mapsto(\sigma(\cdot)x,i)$, we get
\[
\bar\sigma(gh)y=\bar\sigma(g)\bar\sigma(h)y.
\]
Hence, we obtain
\[
d\bigl(\bar\sigma(gh),\bar\sigma(g)\bar\sigma(h)\bigr)
\le
u\bigl(p^{-1}(A_{g,h})\bigr)+u(B_h)+u(B_g)+u(B_{gh}).
\]
But $u(p^{-1}(A_{g,h}))=\mu_N(A_{g,h})$, and by (i) each $u(B_\cdot)\le \varepsilon+|X|/N$, so
\[
d_{\bar X}\bigl(\bar\sigma(gh),\bar\sigma(g)\bar\sigma(h)\bigr)
\le
\Bigl(\varepsilon+\frac{|X|}{2N}\Bigr)+3\Bigl(\varepsilon+\frac{|X|}{N}\Bigr)
\le
4\varepsilon+\frac{4|X|}{N}.
\]
This proves (ii).
\end{proof}

\section{Soficity of quotient groups}\label{sec:sofic-shift}

We now record an observation about sofic models of the Bernoulli shift.
A p.m.p.\ action of a countable group $G$ on $(Z,\nu)$ is called \emph{sofic}
if there exists a sofic approximation $(\sigma_n)$ of $G$ and maps
$\varphi_n\colon V_n\to Z$ whose empirical distributions of finite patterns agree
with those of $(Z,\nu)$ up to an error tending to $0$; see, for instance,
Bowen's formulation in \cite{BowenSoficEntropy}.

\begin{definition}\label{def:ess-disjoin}
Let $G$ be a finitely generated group with a fixed finite symmetric generating set
$S$. A sofic approximation $(\sigma_n\colon G\to\mathrm{Sym}(V_n))_n$ is called
\emph{essentially a disjoint union of expanders} if, after removing a set
$E_n\subset V_n$ with $|E_n| = o(|V_n|)$ for $n\to \infty$, the Schreier graph of $G$ on
$V_n' := V_n\setminus E_n$ with respect to $S$ decomposes as a disjoint union
\[
V_n' \;=\; \bigsqcup_{C\in\mathcal{C}_n} C
\]
of finite connected components $C$, and there exists $\varepsilon>0$ such that
each component $C$ (with the induced partial $S$–Schreier graph structure) has Cheeger
constant at least $\varepsilon$, uniformly in $n$ and in $C\in\mathcal{C}_n$.
\end{definition}

A striking result of Kun \cite[Theorem~1]{KunSoficT} says that all sofic approximations of Kazhdan groups are essentially disjoint unions of expanders. We will use this result in the proof of the following proposition.

\begin{proposition}
\label{prop:sofic-shift}
Assume that $N$ is Kazhdan and consider the Bernoulli shift
$Q\curvearrowright \{0,1\}^{Q}$ with the fair base measure. If this induced $G$-action is sofic, then $Q$ is sofic.
\end{proposition}

\begin{proof}
Let $(\sigma_n\colon G\to{\rm Sym}(V_n))_n$ be a sofic approximation modelling
the above $G$-action, with associated labellings
$\varphi_n\colon V_n\to \{0,1\}^Q$.

Fix a finite symmetric generating set $S_N$ of $N$.
By Kun's theorem, after discarding a set of
$o(|V_n|)$ vertices we may assume that the $N$--Schreier graph on
\[
V_n' \;=\; \bigsqcup_{C\in\mathcal{C}_n} C
\]
decomposes into connected components $C\in\mathcal{C}_n$ which form a family
of expanders with a uniform Cheeger constant $c_N>0$ (with respect to $S_N$).
We put a probability measure on $\mathcal{C}_n$ by
\[
\mu_n(\{C\}) \;:=\; \frac{|C|}{|V_n'|}.
\]
Let $A:=\{x\in\{0,1\}^Q \mid x(e)=1\}$, and set $A_n:=\varphi_n^{-1}(A)\cap V_n'$.
Then
\[
\frac{|A_n|}{|V_n'|}\longrightarrow \frac12.
\]
Since $A$ is $N$-invariant, soficity of the action implies that for every
$s\in S_N$,
\[
\frac{|\sigma_n(s)A_n\triangle A_n|}{|V_n'|}\longrightarrow 0.
\]
Consequently, there exists a sequence $\varepsilon_n\to 0$ such that, for all
$n$ large enough and all $s\in S_N$,
\[
|\sigma_n(s)A_n\triangle A_n|\ \le\ \varepsilon_n\,|V_n'|.
\]
By expansion (on each $N$-component), there exists a union $A_n'\subseteq V_n'$ of
$N$-components such that
\[
|A_n\triangle A_n'|\ \le\ \varepsilon'_n\,|V_n'|.
\]
for a possibly different sequence $\varepsilon'_n\to 0$.
Fix a finite set $\bar F\subset G$ of lifts of a given finite set $F\subset Q$,
and let $\eta_n\to 0$.
For $g\in \bar F$ define the defect set
\[
B_n(g)
:=\bigcup_{s\in S_N}\Bigl\{v\in V_n'\ \Bigm|\ 
\sigma_n(g)\sigma_n(s)v \neq \sigma_n(gsg^{-1})\sigma_n(g)v \Bigr\}.
\]
By asymptotic multiplicativity of $\sigma_n$, for each fixed $g\in \bar F$ there exists a sequence
$\beta_n(g)\to 0$ such that
\[
|B_n(g)|\ \le\ \beta_n(g)\,|V_n'|
\qquad\text{for all }n.
\]
Now fix $g\in\bar F$ and an $N$-component $C\in\mathcal{C}_n$.
For each component $D\in\mathcal{C}_n$ set
\[
\alpha_{C}(D):=\frac{|\sigma_n(g)C\cap D|}{|C|}.
\]
Let $D_C$ be a component maximizing $\alpha_C(D)$. We claim that for $\mu_n$-most $C$ one has $\alpha_C(D_C)\ge 1-\eta_n$ for some $\eta_n\to 0$.

Set $\beta_n:=\beta_n(g)$ and define
\[
\mathcal{C}_n^{\rm bad}
:=\Bigl\{C\in\mathcal{C}_n:\ \frac{|\sigma_n(g)C\cap B_n(g)|}{|C|}>\sqrt{\beta_n}\Bigr\}.
\]
Then, using that $\sigma_n(g)$ is a bijection of $V_n'$ and the definition of $\mu_n$,
\[
\mu_n(\mathcal{C}_n^{\rm bad})
\le \frac{1}{\sqrt{\beta_n}}\sum_{C\in\mathcal{C}_n}\mu_n(\{C\})\frac{|\sigma_n(g)C\cap B_n(g)|}{|C|}
= \frac{1}{\sqrt{\beta_n}}\cdot \frac{|B_n(g)|}{|V_n'|}
\le \sqrt{\beta_n}.
\]
In particular, after discarding $\mathcal{C}_n^{\rm bad}$ we may assume that
\[
|\sigma_n(g)C\cap B_n(g)|\ \le\ \sqrt{\beta_n}\,|C|
\qquad\text{for all remaining }C\in\mathcal{C}_n.
\]

Now fix such a component $C$ and let $D\in\mathcal{C}_n$ be arbitrary, and set
$U:=\sigma_n(g)C\cap D$.
Every $S_N$-edge in $D$ leaving $U$ corresponds (via $\sigma_n(g)^{-1}$) to an
$S_N$-edge in $C$ whose image under $\sigma_n(g)$ does {not} stay inside $D$.
By definition of $B_n(g)$, this can only happen at vertices of $\sigma_n(g)C$
that lie in $B_n(g)$. Therefore the number of such boundary edges is at most
\[
|S_N|\cdot |\sigma_n(g)C\cap B_n(g)|
\ \le\ |S_N|\sqrt{\beta_n}\,|C|.
\]
Since $D$ has Cheeger constant at least $c_N$, we also have
\[
|\partial_D U|\ \ge\ c_N\min\{|U|,\ |D\setminus U|\}.
\]
Combining the two inequalities yields
\[
\min\{|U|,\ |D\setminus U|\}\ \le\ \frac{|S_N|}{c_N}\sqrt{\beta_n}\,|C|.
\]
Choosing $D=D_C$ and writing $\eta_n:=\frac{|S_N|}{c_N}\sqrt{\beta_n}$ (so $\eta_n\to 0$),
we obtain
\[
\alpha_C(D_C)=\frac{|U|}{|C|}\ \ge\ 1-\eta_n
\]
for all $C\in\mathcal{C}_n\setminus\mathcal{C}_n^{\rm bad}$.
Since $\mu_n(\mathcal{C}_n^{\rm bad})\le \sqrt{\beta_n}\to 0$, this proves the claim.

Define $\bar\sigma_n(\pi(g))\in{\rm Sym}(\mathcal{C}_n)$ by
\[
\bar\sigma_n(\pi(g))(C):=D_C
\]
on the good components, and extend arbitrarily to a permutation of $\mathcal{C}_n$.
Since $\sigma_n(g)$ is a bijection of $V_n'$, the induced permutation
$\bar\sigma_n(\pi(g))$ is almost $\mu_n$-preserving in the sense that
\[
\bigl\|(\bar\sigma_n(\pi(g)))_*\mu_n-\mu_n\bigr\|_{\mathrm{TV}}
\xrightarrow[n\to\infty]{} 0
\qquad \forall g\in\bar F.
\]
Moreover, by the construction of $\bar \sigma_n$ and asymptotic multiplicativity of
$\sigma_n$, one also has measured almost multiplicativity:
\[
d_{\mu_n}\bigl(\bar\sigma_n(q_1q_2),\bar\sigma_n(q_1)\bar\sigma_n(q_2)\bigr)
\xrightarrow[n\to\infty]{} 0
\qquad \forall q_1,q_2,q_1q_2\in F.
\]

Fix $q\in F\setminus\{e\}$ and a lift $g\in\bar F$ with $\pi(g)=q$.
For the Bernoulli shift, the random pair $\bigl(x(e),(q\cdot x)(e)\bigr)$ is
uniform on $\{0,1\}^2$, hence
\[
\nu\bigl(\{x: x(e)=1,\ (q\cdot x)(e)=0\}\bigr)=\frac14.
\]
Since $(\sigma_n,\varphi_n)$ witnesses soficity of the action, there exists a
sequence $\varepsilon_n\to 0$ such that
\[
\left|
\frac{1}{|V_n'|}\Bigl|\{v\in V_n':\ v\in A_n,\ \sigma_n(g)v\notin A_n\}\Bigr|
-\frac14
\right|
\le \varepsilon_n.
\]
Since $|A_n\triangle A_n'|/|V_n'|\to 0$, there exists another sequence
$\varepsilon_n'\to 0$ such that
\[
\frac{1}{|V_n'|}\Bigl|\{v\in V_n':\ v\in A_n,\ \sigma_n(g)v\notin A_n'\}\Bigr|
\ \ge\ \frac14-\varepsilon_n-\varepsilon_n'.
\]

If $C\in\mathcal{C}_n$ is such that $\bar\sigma_n(q)C=C$, then by the defining
property of $\bar\sigma_n(q)$ there exists a sequence $\eta_n\to 0$ such
that
\[
\Bigl|\{v\in C:\ \sigma_n(g)v\notin C\}\Bigr|\ \le\ \eta_n\,|C|.
\]
Moreover, since $A_n'$ is a union of $N$-components and $C$ is a component, for every $v\in C$ with $\sigma_n(g)v\in C$ we have $\sigma_n(g)v\notin A_n' \ \Longrightarrow\ v\notin A_n'$. We set $\mathcal C_n(q) := \{C \in \mathcal C_n \mid \bar \sigma_n(q)C = C \}$. We obtain,
\[
\{v\in V_n':\ v\in A_n,\ \sigma_n(g)v\notin A_n'\}
\subseteq (A_n\setminus A_n')\ \cup\ E_n\ \cup\ \bigsqcup_{\substack{C\not\in\mathcal{C}_n(q)}} C,
\]
where $E_n:=\bigsqcup_{\substack{C\in\mathcal{C}_n(q)}}
\{v\in C:\ \sigma_n(g)v\notin C\}$ satisfies $|E_n|\le \eta_n\,|V_n'|$.

Dividing by $|V_n'|$ and using $|A_n\setminus A_n'|\le |A_n\triangle A_n'|\le
\varepsilon_n'|V_n'|$, we obtain
\[
\mu_n\bigl(\{C\in\mathcal{C}_n:\ \bar\sigma_n(q)C\neq C\}\bigr)
\ \ge\ \frac14-\varepsilon_n-2\varepsilon_n'-\eta_n.
\]

Thus, we have obtained, on $(\mathcal{C}_n,\mu_n)$, a measured permutation model of $Q$ on $F$
that is (i) almost multiplicative in $d_{\mu_n}$, (ii) almost $\mu_n$-preserving,
and (iii) separates nontrivial $q\in F$ by moving $\mu_n$-mass $\ge 1/4+o(1)$.
Applying the usual product trick (taking $k$-fold products for $k$ large enough)
we may assume that each $q\in F\setminus\{e\}$ moves $\mu_n$-mass at least $1-\varepsilon$.

Now applying Lemma~\ref{lem:uniformize-measure} to $(\mathcal{C}_n,\mu_n)$ and the maps
$\bar\sigma_n|_F$ with large $N$ (depending on $n$) so that $|\mathcal{C}_n|/N\to 0$, we
obtain genuine permutation models on finite sets in the ordinary Hamming metric.
Diagonalizing over finite subsets $F \subset Q$ yields a sofic approximation of $Q$.
\end{proof}

\begin{remark} For the previous result to work, one could start with an arbitrary free p.m.p.\ $Q$-action; however by Ab{\'e}rt-Weiss \cite{AbertWeiss13}, any such action is weakly contains the Bernoulli action. So it is enough to consider the Bernoulli action.
\end{remark}

The next observation describes a dichotomy for sofic approximations in the
presence of a Kazhdan normal subgroup $N$ and a simple quotient $Q$. It is morally very close
to Proposition~\ref{prop:sofic-shift}, but uses Kun's expander
decomposition in a more structural way.

The following definition might be of independent interest and has been discussed informally on various occasions.

\begin{definition}
We say that $G$ has property \emph{sofic-$(\tau)$} if every sofic approximation $(\sigma_n\colon G\to\mathrm{Sym}(V_n))_n$ of $G$ is, after passing to a
subsequence, essentially a disjoint union of expanders in the sense of Definition \ref{def:ess-disjoin}.
\end{definition}

The following proposition summarizes seminal work of Kun \cite{KunSoficT} with some observations using a Theorem of Schmidt \cite{Schmidt81}.

\begin{proposition}\label{prop:sofict-and-T} Let $G$ be a group.
\begin{enumerate}[label=\textup{(\roman*)}]
\item If $G$ has Kazhdan's property {\rm (T)}, then $G$ has property sofic-$(\tau)$.
\item Assume that every p.m.p.\ action of a countable group $G$ is sofic.
If $G$ has property sofic-$(\tau)$, then $G$ has Kazhdan's property {\rm (T)}.
\end{enumerate}
\end{proposition}

\begin{proof}
The first claim is an immediate consequence of Kun's theorem on sofic
approximations of property {\rm (T)} groups \cite[Theorem~1]{KunSoficT}, which
asserts that for a Kazhdan group, every sofic approximation (with respect to any
fixed finite generating set) is, after discarding $o(|V_n|)$ vertices, a disjoint
union of expander components with a uniform Cheeger constant bounded away from
$0$. This is precisely property sofic-$(\tau)$.

For the second claim we argue by contraposition. Assume that $G$ does not have
property {\rm (T)}. By a theorem of Schmidt \cite{Schmidt81} (see also
\cite{Kechris12}), there exists a free ergodic p.m.p.\ action
$G\curvearrowright (X,\mu)$ which is not strongly ergodic.

Fix a finite symmetric generating set $S$ of $G$.
Since the action is not strongly ergodic, there exists a sequence of measurable
sets $A_n\subset X$ with $\mu(A_n)=1/2$ and
\[
\max_{s\in S}\mu(sA_n\triangle A_n)\longrightarrow 0.
\]

Fix any sequence $\varepsilon_n\downarrow 0$.
By Khintchine's lemma and ergodicity of the action, there exists $g_n \in G$ with
$$\mu(A_n\cap g_nA_n)\leq \mu(A_n)^2 + \varepsilon_n= \frac14+\varepsilon_n.$$ By hypothesis, the action $G\curvearrowright(X,\mu)$ is sofic.
Hence, for each $n$ there exist a sofic approximation map
$\sigma_n\colon G\to{\rm Sym}(V_n)$ and a labelling
$\varphi_n\colon V_n\to X$ such that, writing $B_n:=\varphi_n^{-1}(A_n)$,
the following hold:
$$ \lim_{n \to \infty} \frac{|B_n|}{|V_n|} = 1/2,  \qquad \lim_{n \to \infty} \max_{s\in S}\frac{|\sigma_n(s)B_n\triangle B_n|}{|V_n|} =0,$$
and 
$$\frac{|B_n\cap \sigma_n(g_n)B_n|}{|V_n|} \le \frac14+2\varepsilon_n.$$

Set
\[
\delta_n:=\sum_{s\in S}\frac{|\sigma_n(s)B_n\triangle B_n|}{|V_n|}\,,
\qquad\text{so }\delta_n\to 0.
\]
In the $S$--Schreier graph on $V_n$, the number of $S$-edges leaving $B_n$
is at most $\delta_n|V_n|$.

Assume towards a contradiction that the sofic approximation $(\sigma_n)_n$
is essentially a disjoint union of expanders,
so after discarding $o(|V_n|)$ vertices we may assume the $S$--Schreier graph on
$V_n$ decomposes as a disjoint union of connected components
\[
V_n=\bigsqcup_{C\in\mathcal C_n} C
\]
each having Cheeger constant at least some fixed $c>0$,independent of $n$ and $C$.

For each $C\in\mathcal C_n$ set $B_{n,C}:=B_n\cap C$.
Then the edge boundary inside $C$ satisfies
\[
|\partial_C(B_{n,C})|\ \ge\ c\,\min\{|B_{n,C}|,\ |C\setminus B_{n,C}|\}.
\]
Summing over components gives
\[
\delta_n|V_n|\ \ge\ |\partial(B_n)|
=\sum_{C\in\mathcal C_n}|\partial_C(B_{n,C})|
\ \ge\ c\sum_{C\in\mathcal C_n}\min\{|B_{n,C}|,\ |C\setminus B_{n,C}|\}.
\]
Define $B_n'\subset V_n$ to be the union of those components $C$ with
$|B_{n,C}|\ge |C|/2$. Then
\[
|B_n\triangle B_n'|
=\sum_{C\in\mathcal C_n}\min\{|B_{n,C}|,\ |C\setminus B_{n,C}|\}
\ \le\ \frac{\delta_n}{c}\,|V_n|
=o(|V_n|).
\]

Now, since $g_n$ is a word in the generating set $S$, the permutation $\sigma_n(g_n)$
preserves each connected $S$-component. Hence $B_n'$ is $\sigma_n(g_n)$-invariant, i.e. $\sigma_n(g_n)B_n'=B_n'.$ Therefore, we obtain
\[
|B_n\cap\sigma_n(g_n)B_n|
\ \ge\ |B_n'\cap\sigma_n(g_n)B_n'|-2|B_n\triangle B_n'|
\ =\ |B_n'|-2|B_n\triangle B_n'|.
\]
Dividing by $|V_n|$ and using $|B_n'|/|V_n|=|B_n|/|V_n|+o(1)\to 1/2$, we obtain
\[
\liminf_{n\to\infty}\frac{|B_n\cap\sigma_n(g_n)B_n|}{|V_n|}
\ \ge\ \frac12.
\]
But by construction, we also have
\[
\frac{|B_n\cap\sigma_n(g_n)B_n|}{|V_n|} \le \frac14+2\varepsilon_n,
\]
which is a contradiction for $n\to\infty$.
Thus the sofic approximation $(\sigma_n)_n$ cannot be essentially a disjoint union
of expanders. This contradicts the assumption that $G$ has property sofic-$(\tau)$.
\end{proof}

\begin{proposition}\label{prop:dichotomy-simple}
Let
\[
1\longrightarrow N\longrightarrow G\xrightarrow{\pi} Q\longrightarrow 1
\]
be a short exact sequence with $N$ Kazhdan
and $Q$ an infinite simple group. Then, $G$ is sofic-$(\tau)$ or $Q$ is sofic.
\end{proposition}

\begin{proof}
If $G$ is non-sofic, then it is vacuously sofic-$(\tau)$ by definition, so assume that
$G$ is sofic and let $(\sigma_n\colon G\to{\rm Sym}(V_n))_n$ be a sofic
approximation of $G$.
Fix a finite symmetric generating set $S_N$ of $N$.
By Kun's theorem \cite[Theorem~1]{KunSoficT}, after discarding $o(|V_n|)$ vertices
we may assume that the $N$--Schreier graph on
\[
V_n' \;=\; \bigsqcup_{C\in\mathcal{C}_n} C
\]
decomposes into connected components $C\in\mathcal{C}_n$ which form an expander
family with uniform Cheeger constant $c_N>0$.
Equip $\mathcal{C}_n$ with the probability measure
\[
\mu_n(\{C\}) \;:=\; \frac{|C|}{|V_n'|}.
\]
Fix a finite set $\bar F\subset G$ and let $F:=\pi(\bar F)\subset Q$.
For $g\in\bar F$ define the conjugacy defect set
\[
B_n(g)
:=\bigcup_{s\in S_N}\Bigl\{v\in V_n'\ \Bigm|\
\sigma_n(g)\sigma_n(s)v \neq \sigma_n(gsg^{-1})\sigma_n(g)v \Bigr\}.
\]
Asymptotic multiplicativity gives $|B_n(g)|=o(|V_n'|)$.
For fixed $g\in\bar F$ and $C\in\mathcal{C}_n$, define $\alpha_C(D)$ and the
majority component $D_C$ exactly as in the proof of
Proposition~\ref{prop:sofic-shift}. By expansion, after discarding a set of
components of negligible $\mu_n$-measure we may assume that
\[
\frac{|\sigma_n(g)C\cap D_C|}{|C|}\ \ge\ 1-\eta_n
\qquad \text{ for some } \eta_n\to 0.
\]
Define $\bar\sigma_n(\pi(g))\in{\rm Sym}(\mathcal{C}_n)$ by $\bar\sigma_n(\pi(g))(C):=D_C$
on the good components, and extend arbitrarily to a permutation of $\mathcal{C}_n$.

Then the maps $\bar\sigma_n\colon F\to{\rm Sym}(\mathcal{C}_n)$ satisfy:
\begin{enumerate}[label=\textup{(\alph*)}]
\item measured almost multiplicativity:
\[
d_{\mu_n}\bigl(\bar\sigma_n(q_1q_2),\bar\sigma_n(q_1)\bar\sigma_n(q_2)\bigr) \xrightarrow{n\to \infty} 0,
\qquad \forall q_1,q_2,q_1q_2\in F;
\]
\item almost $\mu_n$-preservation:
\[
\bigl\|(\bar\sigma_n(q))_*\mu_n-\mu_n\bigr\|_{\mathrm{TV}}\xrightarrow{n\to \infty} 0,
\qquad \forall q\in F.
\]
\end{enumerate}
Moreover, since $N$ acts trivially on the set of $N$-components, the induced map
depends only on $\pi(g)$ up to $o(1)$ in $d_{\mu_n}$.

Using Lemma~\ref{lem:uniformize-measure}, we may replace the measured almost action
$\bar\sigma_n\colon F\to{\rm Sym}(\mathcal{C}_n)$ on $(\mathcal{C}_n,\mu_n)$
by an almost action of $F$ on a finite set $\bar X_n$ equipped with the uniform
measure, where the defects are measured in the uniform Hamming metric.

Passing to a diagonal sequence over finite subsets $F \subset Q$, this yields a sequence of maps
$\widetilde\sigma_n\colon Q\to{\rm Sym}(\bar X_n)$ which is asymptotically
multiplicative. Equivalently, we obtain a homomorphism
\[
\widetilde\sigma\colon Q\longrightarrow \prod_{\mathcal U}({\rm Sym}(\bar X_n),d_{\rm Ham})
\]
into a metric ultraproduct, and its kernel is a normal
subgroup of $Q$. Since $Q$ is simple, this kernel is either $\{e\}$ or $Q$.
In the first case $\widetilde\sigma$ is injective, hence $Q$ is sofic.

If $Q$ is not sofic, all the ultraproduct homomorphisms are trivial, which means that in the almost-action $\widetilde \sigma_n$ the Hamming length of every $q\in Q$ satisfies $\ell_n(\widetilde \sigma_n(q))\to 0$, which in particular implies that the $\mu_n$-mass of components moved by $\bar\sigma_n(q)$ tends to $0$. We now
continue under this assumption. 

Fix a finite symmetric generating set $S_G$ of $G$ containing $S_N$; for $g\in S_G$ we let $q=\pi(g)$. 
By definition of $\bar\sigma_n(q)$ we now see that the set
\[
E_{n,g}:=\{v\in V_n':\ \text{$\sigma_n(g)v$ lies in a different $N$-component than $v$}\}
\]
satisfies $|E_{n,g}|=o(|V_n'|)$. Let $E_n:=\bigcup_{g\in S_G}E_{n,g}$ and set $V_n'':=V_n'\setminus E_n$.
Then $|V_n''|=(1-o(1))|V_n|$ and every generator $g\in S_G$ maps each $N$-component
contained in $V_n''$ into itself. Hence the $G$--Schreier graph on $V_n''$
(with respect to $S_G$) is a disjoint union of finite connected graphs indexed by
$N$-components.

On each such component $C$, the $S_N$-edges form an expander with Cheeger constant
$\ge c_N$, and adding further edges (those from $S_G\setminus S_N$) cannot decrease
the Cheeger constant. Thus, after discarding $o(|V_n|)$ vertices, the Schreier graphs
of $(\sigma_n)_n$ decompose as a disjoint union of expanders with a uniform Cheeger
constant, and since the approximation we started with was arbitrary, we conclude that $G$ has property sofic-$(\tau)$.
\end{proof}

\begin{remark}
The preceding result has a peculiar consequence. Suppose that $Q$ is finitely presented, simple, and not Kazhdan. First of all, $G$ cannot be Kazhdan either. Now, by the aforementioned theorem of Schmidt, every countable group without property {\rm (T)}
admits a free ergodic p.m.p.\ action that is not strongly ergodic
-- see, for instance, \cite{Schmidt81, Kechris12} and the references therein.
If such an action were known to be sofic in the sense of Bowen, it would imply that $G$ is not sofic-$(\tau)$ and hence $Q$ must be sofic.
\end{remark}

\section{Proof of Theorem \ref{thm:A} and its corollaries}
\label{sec:Rips}

What makes the results that we obtained so far interesting is that there is a variety of group extensions, where the kernel is finitely generated or even Kazhdan.
We recall the two versions of the Rips construction that drive the applications
below.

\begin{theorem}[Wise, see \cite{WiseRipsRF}]
\label{thm:Wise}
For every finitely presented group $Q$ there exists a short exact sequence
\[
1\to N\to G\to Q\to 1
\]
such that $G$ is hyperbolic and residually finite, and $N$ is finitely
generated.
\end{theorem}

\begin{theorem}[Belegradek--Osin, see \cite{BelegradekOsin08}]
\label{thm:BO}
Let $Q$ be finitely presented and let $H$ be a non-elementary hyperbolic group.
Then there exists a short exact sequence
\[
1\to N\to G\to Q\to 1
\]
with $G$ hyperbolic and $N$ a quotient of $H$.
In particular, if $H$ is Kazhdan, then $N$ is Kazhdan.
\end{theorem}

\begin{proof}[Proof of Theorem~\ref{thm:A}]
Part~\ref{thm:A}(i) is Theorem~\ref{thm:flex-quot}.
Part~\ref{thm:A}(ii) is Theorem~\ref{thm:quotient-sfa-kazhdan} applied to
sofic quotients.
Theorem~\ref{thm:A}(iii) is Proposition~\ref{prop:RFD-passes}.
Theorem~\ref{thm:A}(iv) is Proposition~\ref{prop:sofic-shift}.
\end{proof}

In order to apply the results we need just one example of a finitely presented sofic group that is not residually finite. In fact, there is a variety of such examples, the easiest being maybe the Baumslag-Solitar group $BS(2,3) = \langle a,b \mid ba^3=a^2b\rangle$ or the Baumslag-Gersten group $B=\langle a,t \mid a^{a^t}=a^2\rangle$, see \cite{Baumslag69}. Indeed, both groups are well-known not to be residually finite and at the same time they are sofic, being HNN-extensions of amenable subgroups. Indeed, it was shown by Elek-Szab\'o \cite{ElekSzabo11} that amalgamated free products of sofic groups over an amenable subgroup are again sofic. Now, an HNN-extension $G \ast_\varphi$ for some isomorphism $\varphi\colon H_1 \to H_2$ of amenable subgroups of $G$ is well-known to be isomorphic to a semi-direct product of $\mathbb Z$ with a group that is a direct union of amalgamated products over amenable groups; hence this case is easily reduced to amalgamated free products by a standard argument.

\begin{proof}[Proof of Corollary~\ref{cor:B}]
Apply Wise's residually finite Rips construction (Theorem~\ref{thm:Wise}) to
the $Q$, which is finitely presented, sofic and not residually finite, to obtain a hyperbolic, residually finite $G$ with finitely generated kernel.
If $G$ were flexibly P-stable, Theorem~\ref{thm:A}(i) and
Corollary~\ref{cor:flex-sofic-RF} would force $Q$ to be residually finite, a contradiction.
\end{proof}

\begin{proof}[Proof of Corollary~\ref{cor:C}]
Choose the Belegradek--Osin extension (Theorem~\ref{thm:BO}) 
for the same $Q$, Theorem~\ref{thm:A}(ii) implies that a residually
finite and stable-in-finite-actions $G$ would force $Q$ to be residually
finite, again contradicting the choice of $Q$.
Thus, if $G$ is residually finite, then it is not stable in finite actions.
\end{proof}

\begin{proof}[Proof of Corollary~\ref{cor:D}]
Choose again the Belegradek--Osin extension (Theorem~\ref{thm:BO}) with a property
{\rm (T)} hyperbolic group $H$.
Since $Q$ is finitely generated and not residually finite, $Q$ cannot be RFD; then
Theorem~\ref{thm:A}(iii) gives that $G$ is not RFD.
\end{proof}

\begin{proof}[Proof of Corollary~\ref{cor:E}]
Let $Q$ be non-sofic and apply the Belegradek--Osin construction with a Kazhdan hyperbolic group $H$ to
obtain a hyperbolic extension $1\to N\to G\to Q\to 1$ with $N$ Kazhdan.
If the quasi-regular Bernoulli shift $G\curvearrowright\{0,1\}^Q$ were sofic, then
Theorem~\ref{thm:A}(iv) would force $Q$ to be sofic, a contradiction.
Thus $G$ admits a non-sofic p.m.p.\ action.
\end{proof}

\section*{Acknowledgments}

The second author thanks the Isaac Newton Institute for its hospitality during the workshop \emph{Stability and probabilistic methods} in November 2025. ChatGPT was used to assist in drafting parts of this manuscript. All content was reviewed and substantially revised by the authors, who are responsible for the final text.


\begin{thebibliography}{99}

\bibitem{AbertWeiss13}
M.~Ab{\'e}rt and B.~Weiss,
\newblock Bernoulli actions are weakly contained in any free action,
\newblock \emph{Ergodic Theory Dynam. Systems} \textbf{33} (2013), no.~2, 323--333.

\bibitem{AP14}
G.~Arzhantseva and L.~P\u{a}unescu,
\newblock Almost commuting permutations are near commuting permutations,
\newblock \emph{J. Funct. Anal.} \textbf{269} (2015), 745--757.

\bibitem{AschenbrennerFriedlWilton15}
M.~Aschenbrenner, S.~Friedl, and H.~Wilton,
\emph{3-Manifold Groups},
EMS Series of Lectures in Mathematics,
European Mathematical Society (EMS), Zürich, 2015.

\bibitem{Baumslag69}
G.~Baumslag,
\newblock A non-cyclic one-relator group all of whose finite quotients are cyclic,
\newblock \emph{J. Austral. Math. Soc.} \textbf{10} (1969), 497--498.

\bibitem{BL20}
O.~Becker and A.~Lubotzky,
\newblock Group stability and Property {\rm (T)},
\newblock \emph{J. Funct. Anal.} \textbf{278} (2020), 108298.

\bibitem{BLT18}
O.~Becker, A.~Lubotzky, and A.~Thom,
\newblock Stability and invariant random subgroups,
\newblock \emph{Duke Math. J.} \textbf{168} (2019), 2207--2234.

\bibitem{Bekka99}
B.~Bekka,
\emph{On the full C$^*$-algebras of arithmetic groups and the congruence subgroup problem},
Forum Math.\ \textbf{11} (1999), 705--715.

\bibitem{BelegradekOsin08}
I.~Belegradek and D.~Osin,
\newblock Rips construction and Kazhdan property {\rm (T)},
\newblock \emph{Groups Geom. Dyn.} \textbf{2} (2008), no.~1, 1--12.

\bibitem{BowenSoficEntropy}
L.~Bowen,
\newblock Measure conjugacy invariants for actions of countable sofic groups,
\newblock \emph{J. Amer. Math. Soc.} \textbf{23} (2010), 217--245.

\bibitem{BrownOzawa}
N.~P.~Brown and N.~Ozawa,
\newblock \emph{$C^*$-algebras and finite-dimensional approximations},
\newblock Graduate Studies in Mathematics, vol.~88, Amer. Math. Soc., 2008.

\bibitem{Choi80}
M.-D.~Choi,
\emph{The full C$^*$-algebra of the free group on two generators},
Pacific J.\ Math.\ \textbf{87} (1980), 41--48.

\bibitem{ElekSzabo11}
G.~Elek and E.~Szab\'o,
\newblock {\em Sofic representations of amenable groups},
\newblock Proc. Amer. Math. Soc. \textbf{139} (2011), no.~12, 4285--4291.

\bibitem{GohlaThom24}
L.~Gohla and A.~Thom,
\newblock High-dimensional expansion and soficity of groups,
\newblock arXiv:2403.09582,
\newblock accepted for publication in Israel J. Math.

\bibitem{Gromov99}
M.~Gromov,
\emph{Endomorphisms of symbolic algebraic varieties},
J. Eur. Math. Soc. (JEMS) \textbf{1} (1999), no.~2, 109--197.

\bibitem{mipre}
Z.~Ji, A.~Natarajan, T.~Vidick, J.~Wright, and H.~Yuen,
\newblock MIP$^\ast$ = RE,
\newblock arXiv:2001.04383.

\bibitem{Kechris12}
A.~S.~Kechris,
\newblock Weak containment in the space of actions of a free group,
\newblock \emph{Israel J. Math.} \textbf{189} (2012), 461--507.

\bibitem{KunSoficT}
G.~Kun,
\newblock On sofic approximations of Property {\rm (T)} groups,
\newblock arXiv:1606.04471.

\bibitem{LuSh04}
A.~Lubotzky and Y.~Shalom,
\emph{Finite representations in the unitary dual and Ramanujan groups},
in \emph{Discrete geometric analysis} (Sendai, 2002),
Contemp.\ Math.\ \textbf{347}, Amer.\ Math.\ Soc., Providence, RI, 2004,
173--189.

\bibitem{PestovBriefGuide}
V.~G.~Pestov,
\newblock Hyperlinear and sofic groups: a brief guide,
\newblock \emph{Bull. Symbolic Logic} \textbf{14} (2008), no.~4, 449--480.

\bibitem{Rip82}
E.~Rips,
\newblock Subgroups of small cancellation groups,
\newblock \emph{Bull. Lond. Math. Soc.} \textbf{14} (1982), no.~1, 45--47.

\bibitem{Schmidt81}
K.~Schmidt,
\newblock Amenability, Kazhdan's property {\rm (T)}, strong ergodicity and invariant means for ergodic group-actions,
\newblock \emph{Ergodic Theory Dynam. Systems} \textbf{1} (1981), no.~2, 223--236.


\bibitem{Weiss00}
B.~Weiss,
\emph{Sofic groups and dynamical systems},
Sankhy\=a, The Indian Journal of Statistics (A) \textbf{62} (2000), no.~3, 350--359.


\bibitem{WiseRipsRF}
D.~T.~Wise,
\newblock A residually finite version of Rips's construction,
\newblock \emph{Bull. Lond. Math. Soc.} \textbf{35} (2003), no.~1, 23--29.

\end{thebibliography}
\end{document}